\documentclass[11pt, a4paper]{amsart}
\usepackage[top=35truemm,bottom=35truemm,left=35truemm,right=35truemm]{geometry}
\usepackage{setspace} 
\usepackage{amsmath, amsfonts,amsthm,amssymb,mathrsfs}
\usepackage{enumerate}
\usepackage[dvipdfmx]{graphicx}
\usepackage{ascmac}
\usepackage[arrow,matrix]{xy}
\usepackage{color}
\usepackage{array}
\usepackage{pifont}
\usepackage{longtable}
\usepackage{mathtools}

\newtheorem{theo}{Theorem}
\newtheorem{defi}{Definition}

\newtheorem{prop}{Proposition}
\newtheorem{cor}{Corollary}
\newtheorem{lem}{Lemma}

\renewcommand{\proofname}{\textit{Proof.}}
\renewcommand{\qedsymbol}{$\square$}

\def\C{\mathbb{C}}

\def\H{\mathbb{H}}
\def\R{\mathbb{R}}
\def\E{\mathbb{E}}
\def\U{\mathbb{U}}

\begin{document}

\title[ ]{On 3-dimensional hyperbolic Coxeter pyramids}

\author{Yohei Komori}
\address{Department of Mathematics, School of Education, Waseda University, Nishi-Waseda 1-6-1, Shinjuku, Tokyo 169-8050, Japan}
\email{ykomori@waseda.jp}
\author{Yuriko Umemoto}
\address{Osaka City University Advanced Mathematical Institute, 3-3-18 Sugimoto, Sumiyoshi-ku, Osaka 558-8585, Japan}
\email{yuriko.ummt.77@gmail.com}
\date{}
\thanks{}

\begin{abstract}
After classifying 3-dimensional hyperbolic Coxeter pyramids by means of elementary plane geometry,
we calculate growth functions of corresponding Coxeter groups by using Steinberg formula and conclude that growth rates of them are always Perron numbers.
We also calculate hyperbolic volumes of them and compare volumes with their growth rates. Finally 
we consider a geometric ordering of Coxeter pyramids comparable with their growth rates.
\end{abstract}

\maketitle


\setstretch{1.1}
\section{Introduction}
Let $\H^n$ denote hyperbolic $n$--space and $\overline{\H^n}$ its closure.
A convex polyhedron $P \subset \overline{\H^n}$ of finite volume is called a Coxeter polyhedron if all its dihedral angles are submultiples of $\pi $ or $0$.

Let $G$ be a discrete group generated by finitely many reflections in hyperplanes containing facets of a Coxeter polyhedron $P\subset \overline{\H^n}$ and denote by $S$ the set of generating reflections. 
Then we call $G$ a hyperbolic Coxeter group with respect to $P$.
Note that the Coxeter polyhedron $P$ itself is a fundamental domain of $G$.
For each generator $s \in S$, one has $s^2=1$ while two distinct elements $s, s' \in S$ satisfy either no relation or provide the relation $(ss')^m= 1$ for some integer $m = m(s,s') \geq 2$. 
The first case arises if the corresponding mirrors admit a common perpendicular or intersect at the boundary of $\mathbb{H}^n$, while the second case arises if the mirrors intersect in $\mathbb{H}^n$. 

For a pair $(G, S)$ of a hyperbolic Coxeter group $G$ and a finite set of generators $S$, we can define the word length $\ell_S(x)$ of $x \in G$ with respect to $S$
by the smallest integer $n \geq 0$ for which there exist $s_1, s_2, \ldots, s_n \in S$ such that $x=s_1s_2 \cdots s_n$.
We assume that $\ell_S(1_G)=0$ for the identity element $1_G \in G$.
The growth series $f_S(t)$ of $(G, S)$ is the formal power series $\sum_{k=0}^{\infty} a_kt^k$ where $a_k$ is the number of elements $g \in G$ satisfying $\ell_S(g)=k$.
Since $G$ is an infinite group, the growth rate of $(G, S)$, $\tau =\limsup_{k \rightarrow \infty} \sqrt[k]{a_k}$ satisfies $1\leq \tau \leq |S|$ since $a_k \leq |S|^k$, where $|S|$ denotes the cardinality of $S$.
By means of the Cauchy--Hadamard formula, it turns to the condition $1/|S|\leq R \leq 1$ for the radius of convergence $R$ of $f_S(t)$.
Therefore $f_S(t)$ is not only a formal power series but also an analytic function of $t \in \mathbb{C}$ on the disk $|t|<R$.
It is known to be of exponential growth \cite{dlH1} 
and its growth rate $\tau>1$ is an algebraic integer (cf. \cite{D}). 

In this paper, we classify Coxeter pyramids of finite volume in $\overline{\H^3}$ by using the elementary plane geometry;
we reduce it  to the problem of circles and rectangles in Euclidean plane by using the link of the apex $v=\infty$ of a pyramid 
in upper half space model $\U^3$ (see Section 3.1). 
Note that Tumarkin \cite{T} already classified all Coxeter pyramids of finite volume in $\overline{\H^n}$ 
by using a combinatorial idea.  
Then by means of our former result \cite{KU} essntially,
we show that the growth rates of Coxeter pyramids  in $\overline{\H^3}$ are Perron numbers, i.e.,  real algebraic integers $\tau >1$ all of whose conjugates have strictly smaller absolute values.
We also calculate hyperbolic volumes of Coxeter pyramids and check that volumes  are not proportional to growth rates.
Finally we introduce a partial ordering among Coxeter pyramids by using inclusion relation and show that this ordering is promotional to growth rates.

The paper is organized as follows.
In Section 2 after reviewing the upper half space model of $3$--dimensional hyperbolic space and the link of a vertex of a convex polyhedron, 
we classify Coxeter pyramids  in $\overline{\H^3}$.
In Section 3 we collect some basic results of the growth functions and growth rates of Coxeter groups and show that 
growth rates of Coxeter pyramids are always Perron numbers.
In section 4 we calculate hyperbolic volumes of Coxeter pyramids and compare them with their growth rates.
Finally, in Section 5 we introduce a geometric ordering of Coxeter pyramids comparable with their growth rates.

\section{Coxeter pyramids and their classification}

\subsection{The hyperbolic $3$--space $\U^3$ and its isometry group $I(\U^3)$}
\label{upper-model}
The Euclidean 3-space $(\R^3, |dx|)$ contains the the Euclidean plane
$$\E^2:=\{ x=(x_1,x_2,x_3)\in \R^3 \,|\, x_3=0\}$$
which defines  the upper half space of  $\R^3$
$$\U^3:=\{ x=(x_1,x_2,x_3)\in \R^3 \,|\, x_3>0\}.$$
Then $(\U^3, \frac{|dx|}{x_3}) $ is a model of the hyperbolic $3$--space, so called the {\em upper half space model}.
Let us denote by $\partial \U^3:=\E \cup \{ \infty \}$ the boundary at infinity  of $\U^3$ in the extended space $\widehat{\R}^3:=\R^3 \cup \{ \infty \}$,
and $\overline{\U^3}:=\U^3\cup \partial \U^3$  the closure of $\U^3$ in $\widehat{\R}^3$.

A subset $S \subset \U^3$ is called a {\em hyperplane} of $\U^3$ if it is a Euclidean hemisphere or halfplane orthogonal to $\E^2$. 
If we restrict the hyperbolic metric of $\U^3$ to $S$, then it is a model of the hyperbolic plane.

Any isometry $\varphi$ of $(\U^3, \frac{|dx|}{x_3}) $ can be represented as a composition of reflections with respect to hyperplanes.
Also $\varphi$ extends to the boundary map $\hat{\varphi}$ of $\partial \U^3$ so that $\hat{\varphi}$ is a
composition of reflections with respect to the boundaries of hyperplanes, i.e. Euclidean circles or lines in $\E^2$.
This correspondence induces the isomorphism between the isometry group $I(\U^3)$ of $\U^3$ and M\"obius transformation
group $M(\hat{\C})$ of the Riemann sphere $\hat{\C}$ after identifying $\E^2$ with $\C$ by $z=x_1+ix_2$.
$M(\hat{\C})$ consists of $PSL(2, \C)$ and $J$ the reflection with respect to the real axis of $\C$, defined by $J(z)=\bar{z}$.
It should be remarked that $I(\U^3)$ contains isometries corresponding to rotations, delations centered at $0 \in \C$,
the reflection with respect to the real axis of $\C$, and  parallel translations of $\C$.

\subsection{Convex polyhedra and their vertex links}
A {\em closed half space} $H_S$ is defined by the closed domain of $\U^3$ bounded by a hyperplane $S$.
We define a {\em convex polyhedron} as
a closed domain $P$ of $\U^3$ 
which can be written as the intersection of finitely many closed half spaces;
 (it is also called a {\em finite-sided convex $3$-polyhedron in $\U^3$} in \cite{R}):
$$
P=\bigcap H_S.
$$
In this presentation of $P$, we assume that $F_S:=P \cap S$ is a hyperbolic polygon of $S$.
$F_S$ is called a {\em face} of $P$ and $S$ is called the supporting plane of $F_S$.
The {\em dihedral angle} $\angle ST$ between two faces $F_S$ and $F_T$ is defined as follows:
if the intersection of two supporting planes $S$ and $T$ is nonempty, let us choose a point $x \in S \cap T$ and consider the outer-normal vectors $e_S, e_T \in \R^3$ of $S$ and $T$ 
with respect to $P$ starting from $x$. 
Then the dihedral angle $\angle ST$  is defined by the real number $\theta \in [0, \pi)$ satisfying
$$
\cos \theta =- \langle e_S , e_T \rangle
$$
where $\langle \cdot , \cdot \rangle$ denote the Euclidean inner product of $\R^3$.
It is easy to see that this does not depend on the choice of the base point $x \in S \cap T$.
If $S$ and $T$ are parallel,  i.e. the closures of $S$ and $T$ in  $\overline{\U^3}$ only intersect at a point on $\partial \U^3$,
we define  the dihedral angle $\angle ST$  is equal to zero, while if $S$ and $T$ are ultra-parallel,  i.e. the closures of $S$ and $T$ in  $\overline{\U^3}$ never intersect, 
we do not define  the dihedral angle. 

If the intersection of  two facets $F_S$ and $F_T$ of a convex polyhedron $P$ consists of a geodesic segment, it is called an {\em edge} of $P$.
If the intersection $\bigcap F_S$ of more than two facets is a point, it is called a {\em vertex} of $P$.
If the closures of $F_S$ and $F_T$ in  $\overline{\U^3}$ only intersect at a point on $\partial \U^3$,
it is called an {\em ideal vertex} of $P$.
It should be remarked that the hyperbolic volume of a non-compact convex polyhedron $P$ is finite if and only if 
the closures of $P$ in  $\overline{\U^3}$ consists of $P$ itself and ideal vertices.

Let $\Sigma=S(v; r)$ be the hyperbolic sphere of radius $r$ centered at $v \in \U^3$. 
If we restrict the hyperbolic metric of $\U^3$ to $\Sigma$, then it is a model of the 2-dimensional spherical geometry.
To analyse  local geometry of $P$ at a vertex $v$, it is useful to study a cut locus of $P$ by a hyperbolic sphere of small radius centered at $v$
which is called a {\em vertex link} $L(v)$ of $P$ at $v$:

\begin{theo}{\rm (cf. \cite[Theorem 6.4.1]{R})}
Let $v\in \U^3$ be a vertex of a  convex polyhedron $P$ in $\U^3$ and $\Sigma$ be a hyperbolic sphere of $\U^3$ based at $v$ 
such that $\Sigma$ meets just the facets of $P$ incident with $v$. 
Then the link $L(v) := P \cap \Sigma$ of $v$ in $P$ is a spherical convex polygon in the hyperbolic sphere $\Sigma$. 
If $F_S$ and $F_T$ are two facets of $P$ incident with $v$, then $F_S$ and $F_T$ are adjacent facet of $P$ 
if and only if $F_S\cap \Sigma$ and $F_T \cap \Sigma$ are adjacent sides of $L(v)$. 
If $F_S$ and $F_T$ are adjacent facets of $P$ incident with $v$, then the spherical dihedral angle between $F_S \cap \Sigma$ and $F_T \cap \Sigma$
in $\Sigma$  is equal to the hyperbolic dihedral angle between the supporting hyperplanes $S$ and $T$ in $\U^3$. \label{link1}
\end{theo}

A {\em horosphere} $\Sigma$ of $\U^3$ based at $v \in \partial \U^3$ is defined by a Euclidean sphere in $\U^3$ tangent to $\E^2$ at $v$
when $v \in \E^2$, or a Euclidean plane in $\U^3$ parallel to $\E^2$ when $v=\infty$.
If we restrict the hyperbolic metric of $\U^3$ to $\Sigma$, then it is a model of the Euclidean plane.
Similar to the previous case, it is useful to study a cut locus of $P$ by a small horosphere centered at an ideal vertex $v$
which is also called an {\em ideal vertex link} $L(v)$ of $P$ at $v$ to analyse  local geometry of $P$ at $v$:

\begin{theo}{\rm (cf. \cite[Theorem 6.4.5]{R})}
Let $v\in \partial \U^3$ be an ideal vertex of a  convex polyhedron $P$ in $\U^3$ and $\Sigma$ be a horosphere of $\U^3$ based at $v$ 
such that $\Sigma$ meets just the facets of $P$ incident with $v$. 
Then the link $L(v) := P \cap \Sigma$ of $v$ in $P$ is a  Euclidean convex polygon in the horosphere $\Sigma$.  
If $F_S$ and $F_T$ are two facets of $P$ incident with $v$, then $F_S$ and $F_T$ are adjacent facet of $P$ 
if and only if $F_S\cap \Sigma$ and $F_T \cap \Sigma$ are adjacent sides of $L(v)$. 
If $F_S$ and $F_T$ are adjacent facets of $P$ incident with $v$, then the Euclidean dihedral angle between $F_S \cap \Sigma$ and $F_T \cap \Sigma$
in $\Sigma$  is equal to the hyperbolic dihedral angle between the supporting hyperplanes $S$ and $T$ in $\U^3$. 
\label{link2}
\end{theo}

By means of this theorem we can visualize $L(v)$  of $P$ at $v$ as follows: for an ideal vertex $v \in \E^2$, applying $\varphi \in I(\U^3)$ with
$\hat{\varphi}(z)=\frac{1}{z-v}$ to $P$, we may assume that $v=\infty$.
Since the vertical projection
$$
\nu \;:\; \U^3 \rightarrow \E^2 \;;\; (x_1,x_2,x_3) \mapsto (x_1,x_2)
$$ 
maps a horosphere $\Sigma$ isometrically onto the Euclidean plane $\E^2$,
the vertex link $L(\infty)$ can be realized as a Euclidean polygon of $\E^2$, which will be a key idea for us to classify Coxeter pyramids in the next section.

\subsection{Coxeter pyramids and their classification}
\label{classification}
A convex polyhedron $P$ of finite volume in $\U^3$  is called a {\em Coxeter polyhedron} if all its dihedral angles are submultiples of $\pi $ or $0$.
Any Coxeter polyhedron can be described by a {\em Coxeter graph} defined as follows. 
The nodes of a Coxeter graph correspond to facets of $P$.
Two nodes are joined by an  $m$--labeled edge if the corresponding dihedral angle is equal to 
$\pi/m \;\; (m \geqq 3)$ or $m=\infty$. 
By convention we omit labelings for  $m=3$ and delete edges for $m=2$.
Nodes are connected  by a bold edge when the corresponding facets are parallel in $\U^3$.

Let $P \subset \U^3$ be a Coxeter pyramid, i.e. a Coxeter polyhedron with five facets
consisting of one quadrangular {\em base facet} and four triangular {\em side facets}.
We call the vertex $v$ of $P$ opposite to a base facet an {\em apex} of $P$.
Then the apex $v$ should be an ideal vertex of $P$; suppose that $v$ is in $\U^3$. 
Then by Theorem \ref{link1}, the vertex link $L(v)$ would be isometric to a spherical Coxeter quadrangle,
which contradicts to the absence of spherical Coxeter quadrangles by means of Gauss-Bonnet formula.
Therefore the apex $v$ belongs to $\partial \U^3$ and by Theorem \ref{link2}, the ideal vertex  link of $L(v)$ is isometric to a Euclidean Coxeter quadrangle, that is, a Euclidean rectangle.  

Applying isometries of $\U^3$, we will normalize a Coxeter pyramid $P$ as follows:
we may assume that the apex $v$ is $\infty$, the support plane of the base facet is the hemisphere of radius one centered at the origin of $\E^2$,
and side facets are parallel to $x_1$-axis or $x_2$-axis of $\E^2$.
The vertex link $L(\infty)$ of the apex $\infty$ of $P$ is projected on $\E^2$ by the vertical projection
$\nu \;:\; \U^3 \rightarrow \E^2$,  
and we call its image $\nu (L(\infty))$ the {\em projected link}, which is a rectangle in the closed unit disk in $\E^2$,
whose edges are parallel to $x_1$-axis or $x_2$-axis.
Let us call each edge bounding the projected link $\nu (L(\infty))$ $A, B, C$ and $D$ (see Fig.\ \ref{planeABCD}), and denote by $H_A, H_B, H_C$ and $H_D$ 
the corresponding hyperplanes of $\U^3$.
And let us denote by $H_{b}:=H_{base}$ the base facet.
Then we have
\begin{equation}
\nu(L(\infty)) \subset \nu(\overline{H_b}).
\label{eqABCD}
\end{equation}

\begin{figure}[htbp]
\begin{center}
 \includegraphics [width=170pt, clip]{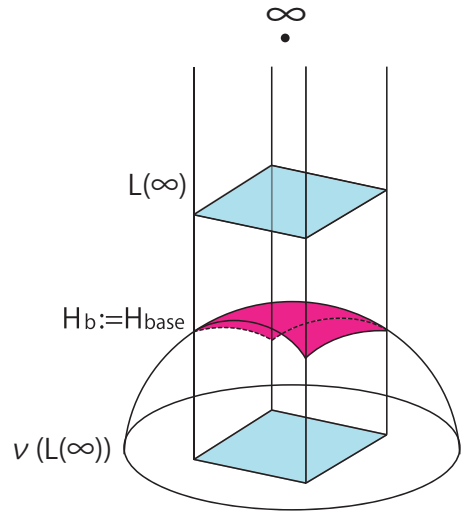}
\end{center}
\caption{A pyramid in $\U^3$ with an apex $\infty$}
\label{upperhalf}
\end{figure} 

 \begin{figure}[htbp]
\begin{center}
 \includegraphics [width=200pt, clip]{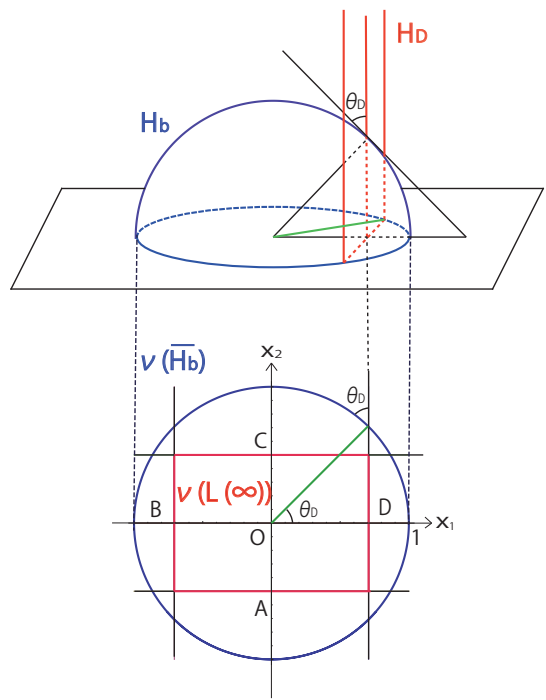}
\end{center}
\caption{The image of the link $L(\infty)$ and the base facet $H_b$ under the projection $\nu$, and the dihedral angle $\theta_D=\angle H_bH_D$}
\label{planeABCD}
\end{figure}  
Let us denote by $\theta_i=\angle H_iH_b$ the dihedral angle between the associated hyperplanes $H_i \subset \U^3$ $(i=A,B,C,D)$ supporting the side facets and the hyperplane $H_b$ supporting the bace facet.
Note that the Euclidean distance $d_{\E^2}(O, i)$ between the origin $O$ and the edge $i$ satisfies $d_{\E^2}(O,i)=\cos \theta_i$(see Fig.\ \ref{planeABCD}).

Applying isometries of $\U^3$ which induce a rotation of $\E^2$ centered at the origin or a reflection of $\E^2$ with respect to a line through the origin,
we may assume that  the dihedral angles of a Coxeter pyramid, say $\theta_A=\pi/k$, $\theta_B=\pi/m$, $\theta_C=\pi/\ell$, $\theta_D=\pi/n$ satisfy
\begin{equation}
k \leq \ell, \;\;\; 
m \leq n, \;\;\;
k \leq m, \;\;\;
\text{and} \;\; \ell \leq n \;\;   \text{when} \;\;   k=m.
\label{klmn}
\end{equation}

Summarising,
\begin{defi}\rm
A Coxeter pyramid $P$ is called {\em normalized} if
\begin{itemize}
\item
the apex of $P$ is $\infty$.
\item
The support plane of the base facet $H_b$ is the hemisphere of radius one centered at the origin of $\E^2$.
\item
Side facets $H_A$ and $H_C$ are parallel to $x_1$-axis 
while side facets $H_B$ and $H_D$ are parallel to $x_2$-axis of $\E^2$.
\end{itemize}
\end{defi}
It is easy to see that there is a unique normalized Coxeter pyramid
in its isometry class.

Now we can classify Coxeter pyramids in terms of
Coxeter graphs by means of conditions (\ref{eqABCD}) and (\ref{klmn}).
It should be remarked that these graphs were first appeared in \cite[Table 2]{T}. 

\begin{theo}
Coxeter pyramids in $\U^3$ can be classfied by Coxeter graphs in {\rm Fig.\ \ref{5gene-pyramid}} up to isometry.

\begin{figure}[htbp]
\begin{center}
 \includegraphics [width=300pt, clip]{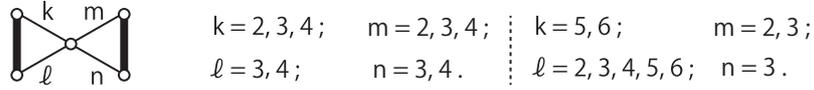}
\end{center}
\caption{Coxeter graphs of Coxeter pyramids  in $\H^3$}
\label{5gene-pyramid}
\end{figure} 
\label{classify}
\end{theo}

\begin{proof}
We determine all $(k, \ell, m, n)$ satisfying conditions (\ref{eqABCD}) and (\ref{klmn}).
From Figure.\ \ref{planepyramid22},  the range of $k$ for $\theta_A=\pi/k$ is $k=2,3,4$.
We can easily find $(k, \ell, m, n)$ satisfying conditions (\ref{eqABCD}) and (\ref{klmn})
from Figures. \ref{planepyramid22},  \ref{planepyramid3} and \ref{planepyramid4}
for $\theta_A=\pi/k$ with $k=2,3$ and $4$ respectively.
\end{proof}

 \begin{figure}[h]
\begin{center}
 \includegraphics [width=200pt, clip]{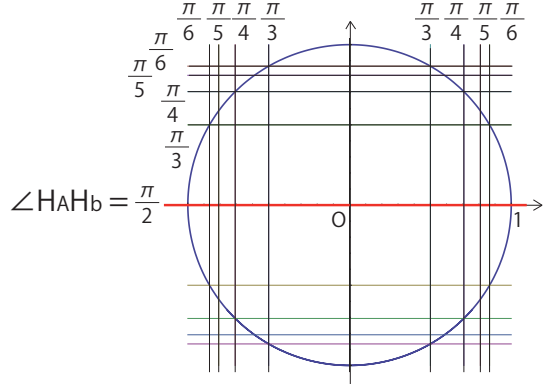}
\end{center}
\caption{The case that $\theta_A=\pi/k=\pi/2$}
\label{planepyramid22}
\end{figure} 

\begin{figure}[htbp]
\begin{center}
 \includegraphics [width=200pt, clip]{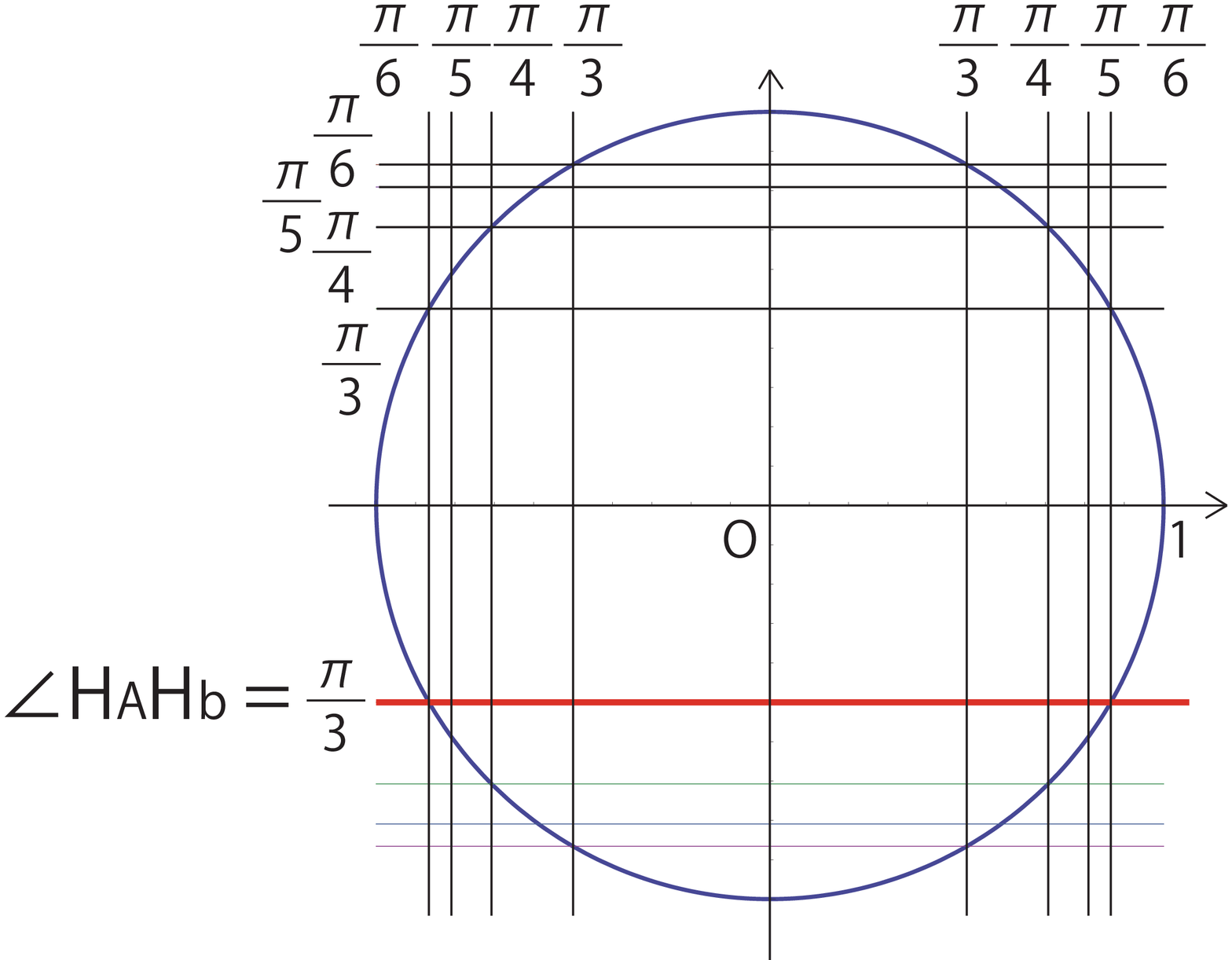}
\end{center}
\caption{The case that $\theta_A=\pi/k=\pi/3$}
\label{planepyramid3}
\end{figure} 

\begin{figure}[htbp]
\begin{center}
 \includegraphics [width=200pt, clip]{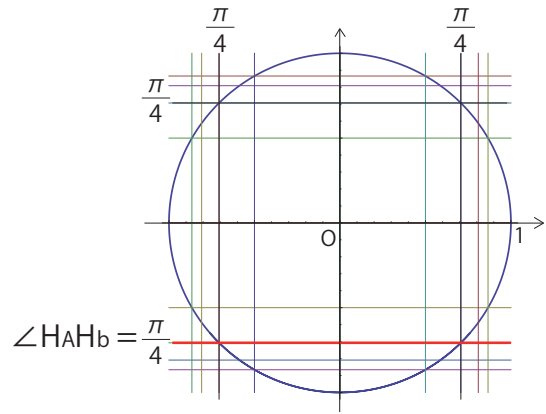}
\end{center}
\caption{The case that $\theta_A=\pi/k=\pi/4$}
\label{planepyramid4}
\end{figure} 

\section{The growth of Coxeter pyramids}
\subsection{Growth functions of Coxeter groups}
From now on, we consider the growth series of a Coxeter group $(G,S)$ where $S$ is its natural set of generating reflections.
In practice, the growth series $f_S(t)$, which is analytic on $|t| < R$ where $R$ is the radius of convergence of $f_S(t)$, extends to a rational function $P(t)/Q(t)$ on $\mathbb{C}$ by analytic continuation where $P(t), Q(t)\in \mathbb{Z}[t]$ are relatively prime. 
There are formulas due to Solomon and Steinberg to calculate the rational function $P(t)/Q(t)$ from the Coxeter graph of $(G,S)$ (\cite{So, St}; see also \cite{Hu}), and we call this rational function the {\em growth function} of $(G,S)$.

\begin{theo}{\rm (Solomon's formula)}\\
The growth function $f_S(t)$ of an irreducible finite Coxeter group $(G, S)$ can be written as 
$f_S(t)=\prod_{i=1}^k[m_i+1]$ where $[n]:=1+t+ \cdots +t^{n-1}$ and $\{m_1, m_2, \ldots, m_k \}$
is the set of exponents of $(G, S)$ {\rm (see \cite{Hu})}.
\end{theo}

We give explicitly the growth functions of irreducible finite Coxeter groups in  Table \ref{tab: exponents}
where we use the notation $[n,m]=[n][m]$ (see \cite{Hu, KP}).

\begin{table}[htbp]
\begin{center}
\begin{tabular}{|c|c|c|}
\hline
Symbol & Exponents & $f_S(t)$\\
\hline
$A_n$ & $1,2, \ldots, n$ & $[2,3, \ldots, n+1]$\\
\hline
$B_n$ & $1,3, \ldots, 2n-1$ & $[2,4, \ldots, 2n]$\\
\hline
$D_n$ & $1,3, \ldots, 2n-3, n-1$ & $[2,4, \ldots, 2n-2][n]$\\
\hline
$E_6$ & $1,4,5,7,8,11$ & $[2,5,6,8,9,12]$\\
\hline
$E_7$ & $1,5,7,9,11,13,17$ & $[2,6,8,10,12,14,18]$\\
\hline
$E_8$ & $1,7,11,13,17,19,23,29$ & $[2,8,12,14,18,20,24,30]$\\
\hline
$F_4$ & $1,5,7,11$ & $[2,6,8,12]$\\
\hline
$H_3$ & $1,5,9$ & $[2,6,10]$\\
\hline
$H_4$ & $1,11,19,29$ & $[2,12,20,30]$\\
\hline
$I_2(m) $ & 1, $m-1$ & $[2,m]$\\
\hline
\end{tabular}
\end{center}
\caption{The growth functions of irreducible finite Coxeter groups}
\label{tab: exponents}
\end{table}
\begin{theo}{\rm (Steinberg's formula)}\\
For a Coxeter group $(G, S)$,
let us denote by $(G_T,T)$ the Coxeter subgroup of $(G, S)$ generated by a subset $T\subseteq S$, and let its growth function be $f_T(t)$.
Set $\mathcal{F}=\{T\subseteq S \;:\; G_T$ is finite $\}$. Then
$$
\frac{1}{f_S(t^{-1})}=\sum _{T \in \mathcal{F}} \frac{(-1)^{|T|}}{f_T(t)}.
$$
\end{theo}

\subsection{Growth rates of Coxeter groups}
It is known that growth rates of hyperbolic Coxeter groups are bigger than $1$ (\cite{dlH1}). 
The formulas of Solomon and Steinberg imply that $P(0)=1$, and hence $a_0=1$ means that $Q(0)=1$.
Therefore the growth rate $\tau >1$ becomes a real algebraic integer. 
If  $t=R$ is the unique pole of $f_S(t)=P(t)/Q(t)$ on the circle $|t|=R$, then $\tau >1$ is a real algebraic integer all of whose conjugates have strictly smaller absolute values.
Such a real algebraic integer is called a {\em Perron number}.

For  two and three--dimensional cocompact hyperbolic Coxeter groups (having compact fundamental polyhedra), Cannon--Wagreich and Parry showed that their growth rates are Salem numbers (\cite{CW, P}),
where a real algebraic integer $\tau>1$ is called a  {\em Salem number} if $\tau^{-1}$ is an algebraic conjugate of $\tau$ and all
 algebraic conjugates of $\tau$ other than $\tau$ and $\tau^{-1}$ lie on the unit circle.
It follows from the definition that a Salem number is a Perron number.

Kellerhals and Perren calculated the growth functions of all four--dimensional cocompact hyperbolic Coxeter groups with at most 6 generators 
and checked numerically that their growth rates are not Salem numbers anymore while
they are Perron numbers (\cite{KP}).

For the noncompact case,
Floyd proved that the growth rates of  two--dimensional cofinte hyperbolic Coxeter groups
are {\em Pisot--Vijayaraghavan numbers}, where a real algebraic integer $\tau>1$ is called a Pisot--Vijayaraghavan number
if all algebraic conjugates of $\tau$ other than $\tau$  lie in the open unit disk (\cite{F}).
A Pisot--Vijayaraghavan number is also a Perron number by definition.

From these results for low-dimensional cases, 
Kellerhals and Perren conjectured that the growth rates  of  hyperbolic Coxeter groups are Perron numbers in general.

\subsection{Calculation of the growth functions}
\label{growth-rates}

Finding all finite Coxeter subgraphs of a Coxeter diagram for a Coxeter pyramid classified in the previous section,
we can calculate growth  functions of  Coxeter pyramids
by means of formulas due to Solomon and Steinberg. 
The denominator polynomials of growth functions of  Coxeter pyramids are listed below,
where $(k,\ell,m,n)$ represents the Coxeter graph in Fig.\ \ref{5gene-pyramid}.
In Table \ref{allvolumes} we also collect  all growth rates of  Coxeter pyramids,
 which arise as the inverse of the smallest positive roots of denominator polynomials of the growth functions.

\begin{enumerate}
 \item $(k,\ell,m,n)=(2,3,2,3): \;(t-1) (t^5+2 t^4+2 t^3+t^2-1)$
 \item $(2,3,2,4): \;(t-1) (t^7+t^6+2 t^5+t^4+2 t^3+t-1)$
  \item $(2,3,2,5):\;(t-1) (t^{13}+t^{12}+2 t^{11}+2 t^{10}+3 t^9+2 t^8+3 t^7+2 t^6+3 t^5+t^4+2 t^3+t-1)$
 \item $(2,3,3,3): \;(t-1) (t^4+2 t^3+t^2+t-1)$
  \item $(2,4,2,4): \;(t-1) (t^4+2 t^3+t^2+t-1)$ 
 \item $(2,3,2,6): \;(t-1) (t^6+2 t^5+t^4+t^3+t^2+t-1)$
 \item $(2,3,3,4): \;(t-1) (t^7+2 t^6+2 t^5+2 t^4+2 t^3+t^2+t-1)$
 \item $(2,4,3,3): \;(t-1) (t^7+2 t^6+2 t^5+3 t^4+2 t^3+t^2+t-1)$
  \item $(2,3,3,5): \;(t-1) (t^{15}+2 t^{14}+3 t^{13}+5 t^{12}+5 t^{11}+7 t^{10}+6 t^9+7 t^8+6 t^7+6 t^6+5 t^5+3 t^4+3 t^3+t-1)$
 \item $(2,3,3,6): \;(t-1)(t^8+2 t^7+3 t^6+3 t^5+3 t^4+2 t^3+t^2+t-1)$
 \item $(2,5,3,3): \;(t-1) (t^9+t^8+2 t^6+t^4+t^3+2 t-1)$
  \item $(2,3,4,4): \;(t-1) (t^5+t^4+t^3+2 t-1)$
   \item $(2,6,3,3): \;(t-1) (2 t^5+t^4+t^3+2 t-1)$
   \item $(2,3,4,5): \;(t-1) (t^{13}+t^{12}+2 t^{11}+2 t^{10}+3 t^9+2 t^8+3 t^7+2 t^6+3 t^5+t^4+3 t^3-t^2+2 t-1)$
  \item $(2,4,3,4): \;(t-1) (t^8+2 t^7+3 t^6+3 t^5+3 t^4+3 t^3+t^2+t-1)$
    \item $(2,3,4,6): \;(t-1) (t^8+2 t^7+3 t^6+4 t^5+3 t^4+3 t^3+t^2+t-1)$
  \item $(2,3,5,5) : \;(t-1) (t^{11}+t^{10}+t^9+2 t^8+t^7+2 t^6+t^5+2 t^4+t^3+2 t-1)$
  \item $(2,3,5,6): \; (t-1) (t^{14}+2 t^{13}+3 t^{12}+4 t^{11}+5 t^{10}+5 t^9+5 t^8+5 t^7+5 t^6+5 t^5+3 t^4+3 t^3+t^2+t-1)$
   \item $(3,3,3,3): \;(t-1) (t^2+2 t-1)$
    \item $(2,3,6,6): \;(t-1) (2 t^6+3 t^5+2 t^4+2 t^3+2 t^2+t-1)$
 \item $(2,4,4,4): \;(t-1) (2 t^4+3 t^3+2 t^2+t-1)$
 \item $(3,3,3,4): \;(t-1) (t^5+2 t^4+t^2+2 t-1)$
  \item $(3,3,3,5): \;(t-1) (t^9+t^8-t^7+3 t^6-t^5+t^4+2 t^3-2 t^2+3 t-1)$
   \item $(3,3,3,6): \:(t-1) (2 t^7+t^6+4 t^5+t^4+3 t^3+2 t-1)$
 \item $(3,3,4,4): \;(t-1) (t^5+2 t^4+t^3+t^2+2 t-1)$
 \item $(3,4,3,4): \;(t-1) (t^6+t^5+2 t^4+t^3+t^2+2 t-1)$
  \item $(3,3,4,5): \;(t-1) (t^9+t^8+2 t^6+3 t^3-2 t^2+3 t-1)$ 
  \item $(3,3,4,6): \;(t-1) (2 t^8+3 t^7+5 t^6+6 t^5+5 t^4+4 t^3+2 t^2+t-1)$ 
 \item $(3,3,5,5): \;(t-1) (t^7+t^6-t^5+2 t^4-t^2+3 t-1)$
 \item $(3,3,5,6):\;(t-1) (2 t^{10}+t^9+2 t^8+t^7+2 t^6+2 t^5+t^4+2 t^3+t^2+2 t-1)$
  \item $(3,3,6,6): \;(t-1) (4 t^5+t^4+2 t^3+t^2+2 t-1)$ 
 \item $(3,4,4,4): \;(t-1) (2 t^6+t^5+2 t^4+2 t^3+t^2+2 t-1)$
 \item $(4,4,4,4): \;(t-1) (4 t^3+t^2+2 t-1)$
 \end{enumerate}
 
\subsection{Arithmetic property of  growth rates}
\label{main-results}
Now we state the main theorem for  growth rates of Coxeter pyramids.

\begin{theo}
The growth rates of  Coxeter pyramids are Perron numbers.
\end{theo}

\begin{proof}
Let us denote the denominator polynomial of $f_S(t)$ by $(t-1)g(t)$.
Then except 4 cases $(3,3,4,5)$, $(2,3,4,5)$, $(3,3,5,5)$, and $(3,3,3,5)$, 
$g(t)$ has a form 
$$
\sum _{k=1}^n b_k t^k -1
$$ 
where $b_k$ is a non-negative integer and the greatest common divisor of $\{k \in \mathbb{N} \; |\; b_k \neq 0 \}$ is $1$.
After multiplying $g(t)$ by $(t+1)$ or $(t+1)^2$, 4 exceptional cases also have this form:
\begin{eqnarray*}
(2,3,4,5) &:&(t + 1) (t^{13}+t^{12}+2 t^{11}+2 t^{10}+3 t^9+2 t^8+3 t^7+2 t^6+3 t^5+t^4+3 t^3\\&&-t^2+2 t-1)\\
&&= t^{14}+2 t^{13}+3 t^{12}+4 t^{11}+5 t^{10}+5 t^9+5 t^8+5 t^7+5 t^6+4 t^5+4 t^4\\&&+2 t^3+t^2+t-1    \\
(3,3,3,5) &:&(t + 1) (t^9 + t^8 - t^7 + 3 t^6 - t^5 + t^4 + 2 t^3 - 2 t^2 + 3 t - 1)\\
&&=  t^{10}+ 2 t^9 + 2 t^7+ 2 t^6+ 3 t^4+ t^2 + 2 t -1   \\
(3,3,4,5) &:&(t + 1) (t^9+t^8+2 t^6+3 t^3-2 t^2+3 t-1)\\
&&=t^{10}+2 t^9+t^8+2 t^7+2 t^6+3 t^4+t^3+t^2+2 t-1    
 \\
(3,3,5,5) &:& (t + 1)^2 (t^7+t^6-t^5+2 t^4-t^2+3 t-1)\\
&&= t^9+3 t^8+2 t^7+t^6+3 t^5+t^4+t^3+4 t^2+t-1.     \\ 
\end{eqnarray*}
In this case we already proved our claim in Lemma 1 of \cite{KU};
for the sake of readability, we recall its proof.
Let us put $h(t)=\sum _{k=1}^n b_k t^k$.
Observe $h(0) = 0, \; h(1)>1$, and $h(t)$ is strictly monotonously increasing on the open interval $(0,1)$.
By the intermediate value theorem, there exists a unique real number $r_1$ in $(0,1)$ such that $h(r_1) = 1$, that is, $g(r_1)=0$.
Since all coefficients $a_k$ of the growth series $f_S(t)$ are non-negative integers, $r_1$ is the radius of convergence of the growth series $f_S(t)$.
Only we have to show is that $g(t)$ has no zeros on the circle $|t|=r_1$ other than $t=r_1$.
Consider a complex number $r_1e^{i\theta}$ on the circle $|t| = r_1$ where $\theta$ is $0\leq \theta <2\pi$.
If we assume $g(r_1e^{i\theta})=0$, that is,  $h(r_1e^{i\theta})=1$,
\begin{equation*}
1=\sum_{k=1}^n b_k r_1 ^k \cos k\theta \leq \sum _{k=1} ^n b_k r_1 ^k=1.
\end{equation*}
This implies that $\cos k\theta=1$ for all $k \in \mathbb{N}$ with $b_k \neq 0$.
Now the assumption that  the greatest common divisor of $\{k \in \mathbb{N} \; |\; b_k \neq 0 \}$ is equal to $1$ concludes that $\theta=0$,
hence $t=r_1$ is the unique pole of $f_S(t)$ on the circle $|t| = r_1$.
\end{proof}

\section{The hyperbolic volumes of Coxeter pyramids}
\label{volumes}

In this section, we calculate hyperbolic volumes of Coxeter pyramids by decomposing them into orthotetrahedra (cf.  \S 10.4. of  \cite{R} ).
For this purpose, we use the projected link which we have introduced in the former section.
We express a Coxeter orthotetrahedron by the symbol $[\pi/\alpha, \pi/\beta, \pi/\gamma]$
if its Coxeter graph is \includegraphics [width=50pt, clip]{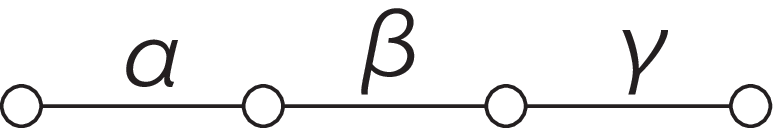},
and also express an ordinary orthotetrahedron by the same symbol $[\theta, \eta, \zeta]$ where $\theta, \eta, \zeta$ are not necessarily submultiples of $\pi$.

First, let us consider the {\em Coxeter orthopyramid} $P=(k,\ell,m,n)=(2,\ell,2,n)$
 whose Coxeter graph is linear (see the Coxeter graph in Fig.\ \ref{orthoscheme2}, 
 where a bullet express the base facet).
Note that its projected link is the rectangle on the first quadrant and its boundary (see the picture of a projected link in Fig.\ \ref{orthoscheme2}).
Then, by the hyperplane passing through the apex $\infty$, origin, and the vertex $u=\bigcap_{i=b,C,D} \overline{H_i}$, $P$ is decomposed into two orthotetrahedra {\textcircled{\footnotesize 1}$[\pi/2-\alpha, \alpha, \pi/\ell]$ and {\textcircled{\footnotesize 2}$[\alpha, \pi/2-\alpha, \pi/n]$ 
with vertex $v=\infty$, where $\displaystyle \alpha=\arctan( \cos(\pi/\ell)/\cos(\pi/n))$ (see the orthotetrahedra with corresponding numbers {\textcircled{\footnotesize 1}, {\textcircled{\footnotesize 2} in Fig.\ \ref{orthoscheme2}).
Note that we omit to write right angles in the pictures of Fig.\ \ref{orthoscheme2}.

\begin{figure}[htbp]
\begin{center}
 \includegraphics [width=300pt, clip]{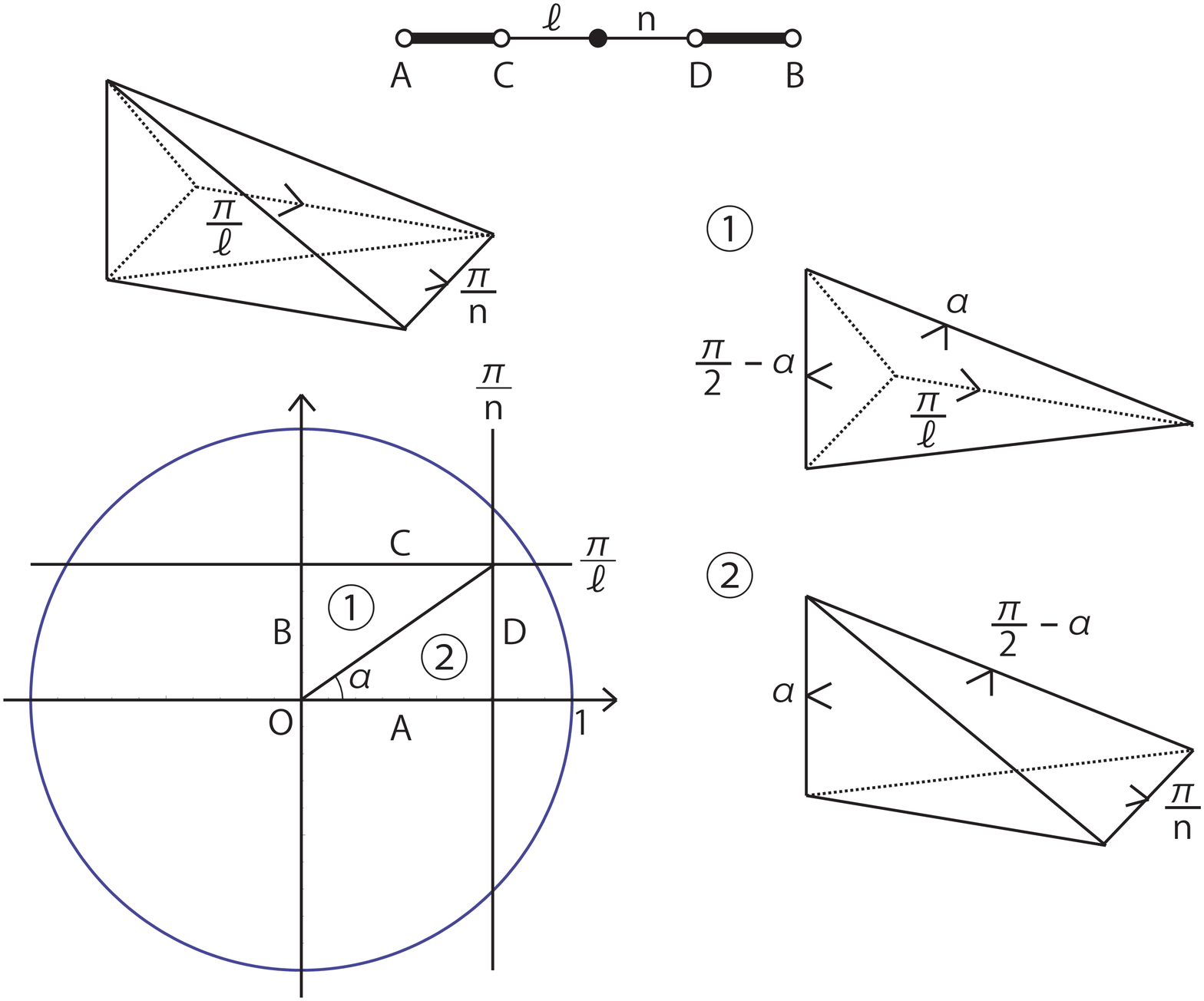}
\end{center}
\caption{Decomposition of a Coxeter orthopyramid into orthtetrahedra}
\label{orthoscheme2}
\end{figure} 
The hyperbolic volume of each orthotetrahedron is calculated as follows (cf. \cite[Theorem 10.4.6]{R}) :
\begin{equation*}
\begin{split}
\mbox{Vol}_{\H^3}[\frac{\pi}{2}-\alpha, \alpha, \frac{\pi}{\ell}]&=
\frac{1}{4}\left( \Pi(\frac{\pi}{2}-\alpha+\frac{\pi}{\ell})+\Pi(\frac{\pi}{2}-\alpha-\frac{\pi}{\ell})+2\Pi(\frac{\pi}{2}-(\frac{\pi}{2}-\alpha))\right)\\
\mbox{Vol}_{\H^3}[\alpha, \frac{\pi}{2}-\alpha, \frac{\pi}{n}]&=
\frac{1}{4}\left( \Pi(\alpha+\frac{\pi}{n})+\Pi(\alpha-\frac{\pi}{n})+2\Pi(\frac{\pi}{2}-\alpha)\right)
\end{split}
\end{equation*}
where $\Pi$ denote {\em Lobachevsky function} defined by the formula of period $\pi$ (cf. \cite[section 10.4]{R}):
$$
\Pi(\theta)=-\int^{\theta}_0 \log |2\sin t| dt.
$$
As a consequence, the hyperbolic volume of the orthopyramid is equal to the sum of volumes of two orthothterahedra: 
$$
\mbox{Vol}_{\H^3}(P)=\mbox{Vol}_{\H^3}[\pi/2-\alpha, \alpha, \pi/\ell]+\mbox{Vol}_{\H^3}[\alpha, \pi/2-\alpha, \pi/n].
$$

Next, let us consider a Coxeter pyramid $P$ which is not an orthopyramid.
By the hyperplanes each of which passes through the apex $\infty$, origin and the vertex $u=\bigcap _{i \in J} \overline{H_i}$, where $J$ is one of the set $\{ b, A, B\}$, $\{ b, B, C\}$, $\{ b, C, D\}$ and $\{ b, D, A\}$ (see section 3.1), $P$ is decomposed into orthopyramids (see Fig.\ \ref{planepyramid5}).
\begin{figure}[htbp]
\begin{center}
 \includegraphics [width=130pt, clip]{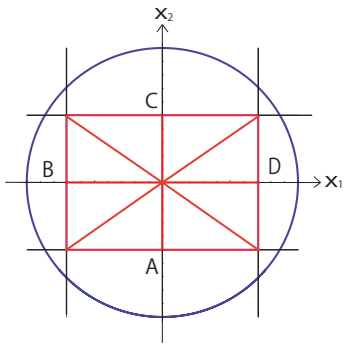}
\end{center}
\caption{Decomposition of a Coxeter pyramid into orthtetrahedra}
\label{planepyramid5}
\end{figure} 
Then, decomposing each orthopyramid into two orthotetrahedra, we have a decomposition of $P$ into orthotetrahedra, and we can calculate the hyperbolic volume of $P$.
We collect  hyperbolic volumes of all Coxeter pyramids in Table \ref{allvolumes}.
We see that $(4,4,4,4)$ has the maximal volume among all Coxeter pyramids in $\overline{\H^3}$.
But we can say much more; by considering the normalization defined in Definition 1,
any pyramid $P$ is contained in an ideal pyramid $Q$, i.e. a pyramid all of whose vertices  are ideal.
Then by means of Theorem 10.4.12 of  \cite{R}, we can calculate the volume of $Q$.
Moreover the proof of 
Theorem 10.4.11 of  \cite{R}  implies that $(4,4,4,4)$, the pyramid obtained by taking the cone to $\infty$ from an ideal square
on the hemisphere of radius one centered at the origin of $\E^2$
has the maximal volume among all ideal pyramids. Therefore

\begin{theo}
A pyramid of maximal volume in $\overline{\H^3}$
is $(4,4,4,4)$.
\end{theo}

It should be remarked that the order of growths is not equal to that of volumes; for example the growth of $(2,3,2,6)$ is smaller than that of $(2,3,3,4)$
while the volume of $(2,3,2,6)$ is bigger than that of $(2,3,3,4)$.

\begin{table}[htbp]
\begin{tabular}{|c|c|c||c|c|c|}
\hline
$(k,\ell,m,n)$&growth&volume&$(k,\ell,m,n)$&growth&volume \\
\hline
$(2,3,2,3)$&$1.73469$&$0.152661$&$( 2,3,5, 6)$&2.40522&$0.75522$ \\
\hline
$(2, 3, 2, 4)$&$1.90648$&$0.25096$&$(3, 3, 3, 3)$&2.41421&$0.610644$\\
\hline
$(2,3,2, 5)$&1.9825&$0.332327$&$( 2,3,6, 6)$&2.42032&$0.845785$\\
\hline
$(2, 3, 3, 3)$&2.06599&$0.305322$&$(2, 4, 4, 4)$&2.45111&$0.915966$\\
\hline
$(2, 4, 2, 4)$& 2.06599&$0.457983$&$(3, 3, 3, 4)$&2.53983&$0.807242$ \\
\hline
$(2,3,2, 6)$&2.01561&$0.422892$&$( 3,3,3, 5)$&2.58553&$0.969976$ \\
\hline
$(2, 3, 3, 4)$&2.1946&$0.403621$&$( 3,3,3, 6)$&2.60198&$1.15111$   \\
\hline
$(2, 4, 3, 3)$&2.23757&$0.501921$&$(3, 3, 4, 4)$&2.64822&$1.00384$ \\
\hline
$(2,3,3, 5)$&2.24692&$0.484988$&$(3, 4, 3, 4)$&2.65364&$1.11256$ \\
\hline
$(2,3,3, 6)$&2.26809&$0.575553$&$(3,3,4, 5)$&2.68684&$1.16657$ \\
\hline
$(2, 5, 3, 3)$&2.30482&$0.664654$&$( 3,3,4, 6)$&2.70039&$1.34771$ \\
\hline
$(2, 3, 4, 4)$&2.30522&$0.501921$&$(3,3,5, 5)$&2.72275&$1.32931$\\
\hline
$(2, 6, 3,3)$&2.33081&$0.845785$&$( 3,3,5, 6)$&2.73526&$1.51044$ \\
\hline
$(2,3,4, 5)$&2.34913&$0.583287$&$( 3,3,6, 6)$&2.74738&$1.69157$\\
\hline
$(2, 4, 3, 4)$&2.35204&$0.708943$&$(3, 4, 4, 4)$&2.75303&$1.41789$ \\
\hline
$(2,3,4, 6)$&2.36644&$0.673853$&$(4, 4, 4, 4)$&2.84547&$1.83193$\\
\hline
$(2,3,5, 5)$&2.38946&$0.664654$&&&\\
\hline
\end{tabular}
\caption{Hyperbolic volumes of all Coxeter pyramids in $\overline{\H^3}$}
\label{allvolumes}
\end{table}

\section{A geometric order of Coxeter pyramids comparable with their growth rates}
We define a natural order of Coxeter pyramids in a geometric way.
For Coxeter pyramids $P_1$ and $P_2$, we define $P_1 \leq P_2$ if there exists $\varphi \in I(\U^3)$ such that
$\varphi(P_1) \subset P_2$.
It is easy to see that this binary relation is a partial order of the set of Coxeter pyramids.
We also denote $G_1 \leq G_2$ if corresponding  Coxeter pyramids satisfy $P_1 \leq P_2$.
We explain this geometric order in terms of $(k, \ell, m, n)$
appeared in section \ref{classification}.

We study the relationship between this partial order and the growth rates of reflection groups.
First, we shall state the following lemma  which makes us to get back this order of reflection groups to that of the projected links of the apex $\infty$ of their fundamental Coxeter pyramids. 

\begin{lem}
Let $G_1$ and $G_2$ be Coxeter groups defined by Coxeter pyramids.
Then $G_1 \leq G_2$ if and only if there exist corresponding  Coxeter pyramids $P_1$ and $P_2$ with apex $\infty$ and bounded by a common hemisphere of radius $1$ centered at the origin, such that $P_1 \subseteq P_2$.
Furthermore,  the projected links $\Delta_1:=\nu(L_1(\infty))$ and $\Delta_2:=\nu(L_2(\infty))$ on $\E^2$ of $P_1$ and $P_2$ satisfy $\Delta_1 \subseteq \Delta_2$ in this case.
\label{order}
\end{lem}

\begin{proof}
As we have already seen in Section 2, there exists a  Coxeter pyramid $P_i$ $(i=1,2)$ of $G_i$ transferred to have the apex $\infty$ and bounded by a common hemisphere of radius $1$ centered at the origin. 
Therefore  by means of concurrency of the side facets of $P_1$ and $P_2$, $P_1\subseteq P_2$ if and only if $\Delta_1 \subseteq \Delta_2$.
\end{proof}

\begin{prop}
For Coxeter groups $G_1$ and $G_2$ defined by Coxeter pyramids, $G_1 \leq G_2$ if and only if there exist corresponding Coxeter pyramids $P_1$ and $P_2$
with numbering $(k', \ell', m', n')$ and $(k'', \ell'', m'', n'')$ described  in {\rm Fig.\ \ref{5gene-pyramid}}
satisfying $k'\leq k''$, $\ell' \leq \ell''$, $m'\leq m''$, and $n'\leq n''$.
\end{prop}

\begin{proof}
It is obvious from concurrency of side facets passing through the apex $\infty$.
\end{proof}

We describe the order relations of   Coxeter pyramids in Fig.\ \ref{inclusion};
each number $(k, \ell, m, n)$ represents the corresponding Coxeter pyramid and we attach its growth rate also.
From this chart, we can see the following result.

\begin{figure}[h]
\begin{center}
 \includegraphics [width=450pt, clip]{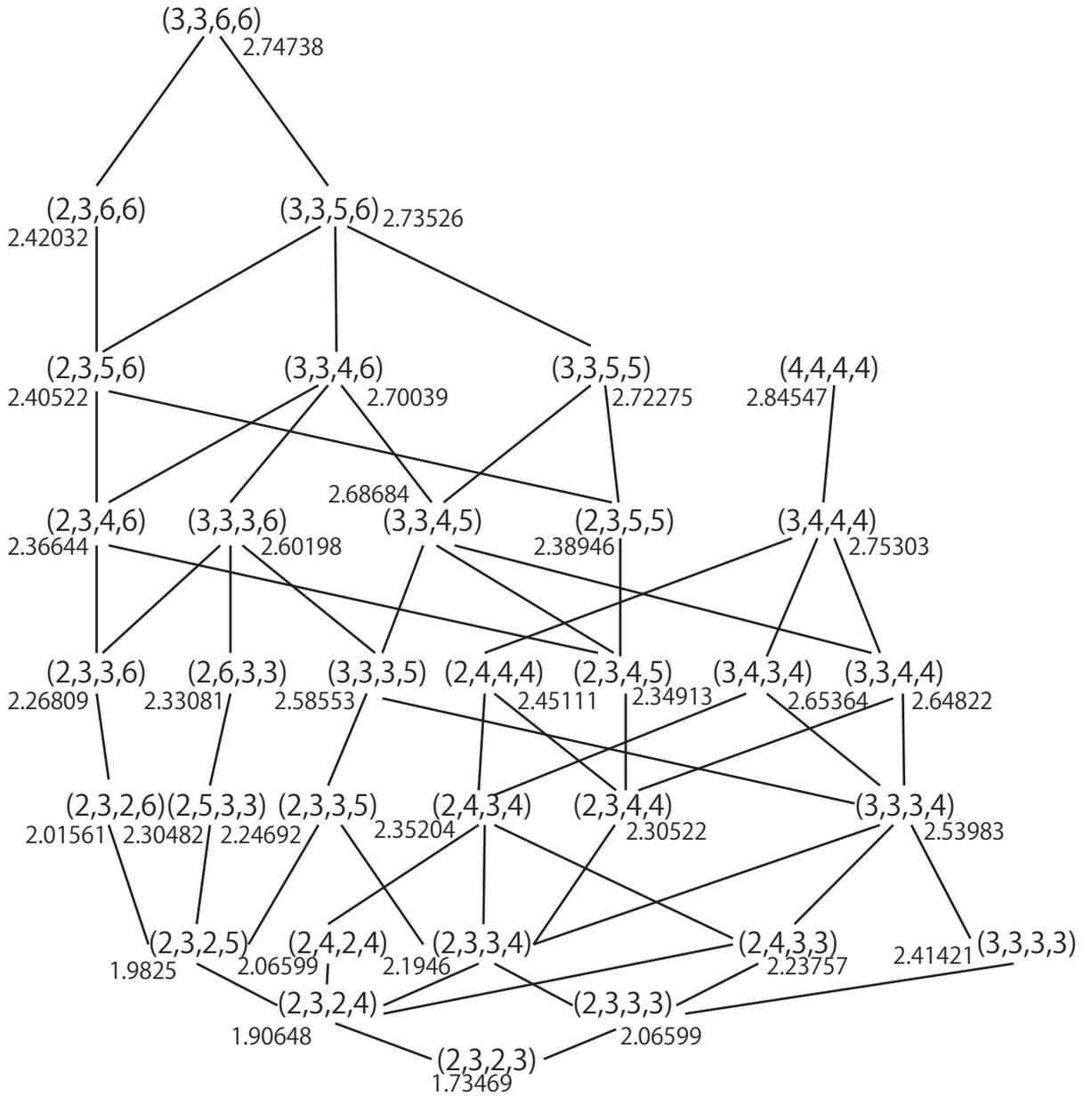}
\end{center}
\caption{A partial order of Coxeter pyramids comparable with their growth rates}
\label{inclusion}
\end{figure}

\begin{cor}
Let $G_1$ and $G_2$ be two Coxeter groups defined by Coxeter pyramids,  and $\tau_1$ and $\tau_2$ be their growth rates respectively.
Then $\tau_1\leq \tau_2$ if and only if $G_1\leq G_2$.
\label{cor-growth}
\end{cor}

Finally it shoud be remarked that the hyperbolic volume of the defining Coxeter pyramid of $G_1$ is smaller than or equal to that of $G_2$ 
if $G_1\leq G_2$ from the definition, while the converse is not true in general.

\clearpage
\section{Acknowledgement}
The authors thank Professor Ruth Kellerhals for telling them the reference \cite{T}.
The first author was partially supported by Grant-in-Aid for Scientific Research(C) (19540194), Ministry of Education, Science and Culture of Japan.
The second author was partially supported by Grant-in-Aid for JSPS Fellows 
no.\ 12J04747, and by the JSPS Institutional Program for 
Young Researcher Overseas Visits
`` Promoting international young researchers in mathematics and 
mathematical sciences led by OCAMI ".

\end{document}